\documentclass{amsart}
\usepackage{amssymb}
\usepackage{amsmath}
\usepackage{diagrams}

\diagramstyle[centredisplay,dpi=600,UglyObsolete,heads=LaTeX]

\newcommand\Zset{\mathbb {Z}}
\newcommand\Qmax{Q^r_{\mathrm{max}}}
\newcommand\Qlmax{Q^l_{\mathrm{max}}}
\newcommand\Qsimmax{Q^{\sigma}_{\mathrm{max}}}
\newcommand\Qtot{Q^r_{\mathrm{tot}}}

\newcommand\Qlrtot{Q_{\mathrm{tot}}}
\newcommand\Qsimtot{Q^{\sigma}_{\mathrm{tot}}}
\newcommand\Qcl{Q^r_{\mathrm{cl}}}
\newcommand\Qlrcl{Q_{\mathrm{cl}}}
\newcommand\f{{\mathcal F}}
\newcommand\te{{\mathcal T}}
\newcommand\pe{{\mathcal P}}
\newcommand\ef{{\mathfrak F}}
\newcommand\eg{{\mathfrak G}}
\newcommand\de{{\mathfrak D}}
\newcommand\dirlim{\mathop{\varinjlim}\limits}
\newcommand\ce{{\mathcal C}}

\newcommand\tor{\mathrm{Tor}}
\newcommand\homo{\mathrm{Hom}}
\newcommand\ann{{\mathrm{ann}}}
\newcommand\coker{{\mathrm{coker}}}

\newtheorem{theorem}{Theorem}[section]
\newtheorem{lemma}{Lemma}[section]
\newtheorem{corollary}{Corollary}[section]
\newtheorem{proposition}{Proposition}[section]

\theoremstyle{definition}
\newtheorem{definition}{Definition}[section]
\newtheorem{example}{Example}[section]

\begin{document}

\title{Perfect Symmetric Rings of Quotients}

\author{Lia Va\v s}

\address{Department of Mathematics, Physics and Statistics,
University of the Sciences in Philadelphia, 600 S. 43rd St.,
Philadelphia, PA 19104}

\email{l.vas@usp.edu}

\begin{abstract}
Perfect Gabriel filters of right ideals and their corresponding right rings of quotients have the desirable feature that every module of quotients is determined solely by the right ring of quotients. On the other hand, symmetric rings of quotients have a symmetry that mimics the commutative case. In this paper, we study rings of quotients that combine these two desirable properties.

We define the symmetric versions of a right perfect ring of quotients and a right perfect Gabriel filter -- the perfect symmetric ring of quotients and the perfect symmetric Gabriel filter and study their properties. Then we prove that the standard construction of the total right ring of quotients $\Qtot(R)$ can be adapted to the construction of the largest perfect symmetric ring of quotients -- the total symmetric ring of quotients $\Qsimtot(R).$ We also demonstrate that Morita's construction of $\Qtot(R)$ can be adapted to the construction of $\Qsimtot(R).$
\end{abstract}

\keywords{Torsion Theory; Gabriel Filter; Ring of Quotients; Perfect; Total.}

\subjclass[2000]{ 16S90, 
16N80} 

\maketitle

\section{Introduction}

As T.Y. Lam puts it in \cite{Lam},  noncommutative localization is ``The Good, the Bad and the Ugly''. It is ``good'' since for a class of noncommutative rings we can mimic the construction of a quotient field of an integral domain and embed those rings in division rings. It is "bad`` since not all rings -- not even all domains -- can be embedded in a division ring. Finally, it is "ugly`` because even those rings that can have such an embedding might have nonisomorphic division hulls.

Different types of rings of quotients have been used in attempts to rectify the ''badness``. For example, the maximal right ring of quotients $\Qmax(R)$ exists for every ring $R$ and has properties that bring $\Qmax(R)$ closer to being a division ring than $R$. However, $\Qmax(R)$ may fail to have some features of the classical ring of quotients $\Qcl(R)$ that we would prefer to keep.

Gabriel's work in \cite{Gabriel} provides the basis for a uniform treatment of different rings of quotients. Presently, torsion theories provide a framework suitable for a comprehensive study of rings of quotients. Starting with a hereditary torsion theory $\tau$ or, equivalently, with a Gabriel filter of right ideals $\ef,$ one can define the module of quotients $M_{\ef}$ for every right $R$-module $M.$ For $M=R,$ we obtain the right ring of quotients $R_{\ef}.$ Using this approach, $\Qmax(R)$ is the right ring of quotients with respect to the filter of dense right ideals.

If a Gabriel filter $\ef$ is such that all the modules of quotients $M_{\ef}$ are determined solely by the right ring of quotients $R_{\ef}$ by
\[M_{\ef}\cong M\otimes_R R_{\ef},\] then $\ef$ and $R_{\ef}$ are called perfect. This feature makes perfect filters interesting to work with. Perfect right rings of quotients and perfect filters have been studied and characterized in more details (see Thm. 2.1, p. 227 in \cite{Stenstrom} and Prop. 3.4, p. 231 in \cite{Stenstrom}).
The largest perfect right ring of quotients is another attempt to find a ''good, not-bad, not-ugly`` ring of quotients. Indeed, for every ring the largest perfect right ring of quotients exists, is denoted by $\Qtot(R),$ and is called the
total right ring of quotients.  It is a subring of $\Qmax(R)$ and, if $R$ is right Ore, it contains $\Qcl(R).$

Noncommutative localization might have another ''bad`` property that classical localization of an integral domain does not have - non symmetry. Namely, left localization might not yield the same quotient as right localization. The (one-sided) noncommutative localization has another downfall. Namely, if $R$ and $S$ are commutative rings, $\ef$ a right Gabriel filter on $R,$ $S$ a right ring of quotients with respect to some filter $\ef'$ and $f:R\rightarrow S$ a homomorphism such that $f(I)S\in \ef'$ for all $I\in \ef,$ then $f$ extends uniquely to a ring homomorphism $R_{\ef}\rightarrow S.$ However, this property might fail in noncommutative case (see Example 4.29, p 98 in \cite{Ortega_thesis}).

As a result, the symmetric version of rings of quotients has begun to attract ring theorists. The symmetric Martindale ring of quotients (see 14B in \cite{Lam}) and the maximal symmetric ring of quotients $\Qsimmax(R)$ (introduced in \cite{Utumi}, studied in \cite{Lanning} and \cite{Ortega_paper_Qsimmax}) emerged. Schelter \cite{Schelter} generalized the commutative case by working over the ring $R\otimes_{\Zset}R^{op}$ and constructed symmetric rings of quotients without the downfall of the noncommutative case. Schelter's work on symmetric rings of quotients parallels Gabriel's work on right rings of quotients -- it provides the basis for a uniform treatment of two-sided rings of quotients using torsion theories. Namely, if $\ef_l$ and $\ef_r$ are left and right Gabriel filters, the symmetric ring of quotients $_{\ef_l}R_{\ef_r}$ can be defined.
Using this approach, $\Qsimmax(R)$ is a two-sided localization with respect to left and right dense ideals. In \cite{Ortega_paper} and \cite{Ortega_thesis}, Ortega defines the symmetric modules of quotients $_{\ef_l}M_{\ef_r}$ with respect to left and right Gabriel filters $\ef_l$ and $\ef_r$ of $R$-bimodules $M.$ If $M=R,$ Ortega obtains Schelter's $_{\ef_l}R_{\ef_r}.$

In this work, we define the symmetric version of a right perfect ring of quotients and the total right ring of quotients. Specifically, if $\ef_l$ and $\ef_r$ are left and right Gabriel filters and $_{\ef_l}R_{\ef_r}$ the symmetric ring of quotients, we study the conditions for the symmetric filter induced by $\ef_l$ and $\ef_r$ to have the property
\[_{\ef_l}M_{\ef_r}\cong\, _{\ef_l}R_{\ef_r}\otimes_R M\otimes_R\,_{\ef_l}R_{\ef_r}\]
for all $R$-bimodules $M.$ We characterize the perfect symmetric rings of quotients and perfect symmetric Gabriel filters proving the symmetric version of Theorem 2.1, p. 227 in \cite{Stenstrom} and Proposition 3.4, p. 231 in \cite{Stenstrom}. Then, we prove that every ring has the largest perfect symmetric ring of quotients -- the total symmetric ring of quotients $\Qsimtot(R).$ We prove that Morita's construction of $\Qtot(R)$ can be adapted to the construction of $\Qsimtot(R).$

In section \ref{section_right_quotients}, we review some preliminaries of one-sided torsion theories and rings of quotients. We review the characterization of perfect right rings of quotients and perfect right filters we attempt to adapt to the symmetric case.

In section \ref{section_symmetric_quotients}, we review the symmetric version of notions from section \ref{section_right_quotients} and prove that the symmetric version of tensoring torsion theory has all the properties the one-sided version of tensoring theory had (Proposition \ref{symmetric_tensor}).

In section \ref{section_symmetric_perfect}, we present the characterization of perfect symmetric rings of quotients (Theorem \ref{PerfectSymmetricQuotient}) and perfect symmetric filters (Theorem \ref{perfect_symmetric_filter}).

In section \ref{section_total_symmetric}, we define the perfect symmetric ring of quotients $\Qsimtot(R)$ (Theorem \ref{total_symmetric_quotient}). We prove that $\Qsimtot(R)$ has the symmetric version of properties that are satisfied by $\Qtot(R).$ We study $\Qsimtot(R)$ for some specific classes of rings and illustrate our work with some examples.

The total right ring of quotients $\Qtot(R)$ is usually constructed as a directed union of all
perfect right rings of quotients. Morita in \cite{Morita3} has a different
approach. He starts from the
maximal right ring of quotients $\Qmax(R)$ to construct
$\Qtot(R)$ by transfinite induction on ordinals, "descending" from
$\Qmax(R)$ towards $R$ instead of "going upwards" starting from
$R$ using the directed family as in the classical construction.
In section \ref{section_Morita_symmetric}, we prove that the symmetric version of this construction produces the total symmetric ring of quotients $\Qsimtot(R)$ (Theorem \ref{Moritas_construction}).

We finish the paper with a list of open problems in section \ref{section_questions}.

\section{Preliminaries on one-sided rings of quotients}
\label{section_right_quotients}

In this paper, a ring is an associative ring with identity. If $M$ is a right $R$-module with submodule $N$ and $m\in M,$  $(m: N)$ denotes $\{r\in R\, | \, mr\in N\}.$ Similarly, if $M$ is a left $R$-module, $(N:m)$ denotes $\{r\in R\, | \, rm\in N\}.$

Throughout this paper, we shall use the definition of torsion theory, hereditary and faithful torsion theory, and Gabriel filter  as given in \cite{Stenstrom} or \cite{Bland_book}. $\tau = (\te, \f)$ denotes that $\tau$ is a torsion theory with torsion class $\te$ and torsion-free class $\f.$ If $M$ is a right $R$-module, $\te(M)$ is the torsion submodule of $M$ and $\f(M)$ is the torsion-free quotient $M/\te(M).$
We can consider the notion of torsion theory for right or left $R$-modules. To emphasize if we are referring to right or left version of these notions, we shall be referring to right and left torsion theories.

If $\tau$ is a hereditary torsion theory with Gabriel filter $\ef$ and $M$ is a right $R$-module, the module of quotients $M_{\ef}$ of $M$ is defined by
\[M_{\ef}=\dirlim_{I\in\ef}\homo(I, \frac{M}{\te(M)})\]
(see chapter IX in \cite{Stenstrom}).
The $R$-module $R_{\ef}$ has a ring structure and is called the
right ring of quotients with respect to $\tau.$ $M_{\ef}$ has a structure of a right
$R_{\ef}$-module (see pages 195--197 in \cite{Stenstrom}). If $\ef$ is a Gabriel filter of left ideals and $M$ a left $R$-module we denote the module of quotients of $M$ with respect to $\ef$ by $_{\ef}M.$

The localization map $q_M:M\rightarrow M_{\ef}$ (the composition of
$M\twoheadrightarrow M/\te(M)$ with $M/\te(M)\hookrightarrow M_{\ef}$) defines a left exact functor $q$ from the category of right $R$-modules to the category of right $R_{\ef}$-modules (see \cite{Stenstrom} p. 197--199).

We recall some important examples of torsion theories.
\begin{example}
(1) The largest hereditary faithful torsion is called the {\em Lambek torsion theory}.
Its Gabriel filter is the set of
all dense right ideals (see Proposition VI 5.5, p. 147 in
\cite{Stenstrom}) and its right ring
of quotients is the maximal right ring of quotients $\Qmax(R)$ (see sections 13B and 13C in \cite{Lam} and
Example 1, p. 200, \cite{Stenstrom}).

(2) If $R$ is a right Ore ring, the  {\em classical torsion theory} is a
hereditary and faithful torsion theory defined by
$\te(M)=\{m\in M|mt=0$ for some nonzero regular element $t\}$ for a right
$R$-module $M.$ The classical right ring of quotients  $\Qcl(R)$ is the corresponding right ring of quotients (Example 2, p. 200, \cite{Stenstrom}).

(3) Let $f:R\rightarrow S$ be a ring homomorphism. The collection of all $R$-modules $M$ such that $M\otimes_R S = 0$ is closed under quotients, extensions and direct sums. Moreover, if $f$ makes $S$ into a flat left $R$-module, then this
collection is closed under submodules and, hence, defines a
hereditary torsion theory. We denote this torsion
theory by $\tau_S$ and its torsion class by $\te_S.$
If $f$ makes $S$ into a flat right $R$-module, we can define a left hereditary torsion theory by the condition that $M$ is torsion iff $S\otimes_R M=0.$ We denote this torsion
theory by $_S\tau.$
\end{example}

We note an easy lemma related to Example (3) above.

\begin{lemma} Let $f$ be a ring homomorphism that makes $S$ into a flat left $R$-module.
\begin{enumerate}
\item For any $M,$ $\te_S(M)=\ker( i_M: M\rightarrow M\otimes_R S).$ $M\in\te_S(M)$ iff $i_M: M\rightarrow M\otimes_R S$ is 0-map.

\item The sequence below is exact.
\begin{diagram}
0&\rTo& \te_S(M)&\rTo& M&\rTo &M\otimes_R S&\rTo &M\otimes_R S/f(R)&\rTo &0
\end{diagram}

\item If $f$ is a monomorphism, all flat right $R$-modules are $\tau_S$-torsion-free and $\tau_S$ is faithful.
\end{enumerate}
\label{properties_of_onesided_tensor}
\end{lemma}
\begin{proof}
(1) From the definition of $\tau_S$ it follows that the largest torsion submodule of $M$ is the kernel of the map $i_M.$ For the second part of the claim, one direction is obvious. The other holds since $m\otimes 1=0$ implies that $m\otimes s=0$ for all $s\in S.$

(2) The exactness of the sequence $0\rightarrow \ker i_M \rightarrow M\rightarrow M\otimes_R S\rightarrow \coker i_M \rightarrow 0$ proves the claim since $\ker i_M=\te_S(M)$ and $\coker i_M=M\otimes_R S/f(R).$

(3) Tensoring the exact sequence $0\rightarrow R\rightarrow S\rightarrow S/f(R)\rightarrow0$ with $M$ from the left produces the exact sequence $0\rightarrow \tor_1^R(M, S/R)\rightarrow M\otimes_R R \rightarrow M\otimes_R S\rightarrow M\otimes_R S/f(R) \rightarrow 0.$ The isomorphism $M\cong M\otimes_R R,$ gives us the isomorphism $\te_S(M)\cong \tor_1^R(M, S/f(R)).$ Thus, all flat right $R$-modules are $\tau_S$ torsion-free.
\end{proof}

Recall that the ring homomorphism $f:R\rightarrow S$ is
called a {\em ring epimorphism} if for all rings $T$ and
homomorphisms $g,h: S\rightarrow T,$ $gf = hf$ implies $g=h.$ The map $f:R\rightarrow S$ is a ring epimorphism iff the canonical map $S\otimes_R S\rightarrow S$ is bijective or, equivalently, if $S/f(R)\otimes_R S=S\otimes_R S/f(R)=0$ (see Proposition XI 1.2, p. 226 in \cite{Stenstrom}). A ring epimorphism $f$ necessarily maps the identity $1_R$ of $R$ to the identity $1_S$ of $S.$ To see this assume that $f(1_R)=s.$ This implies that the right multiplication by $s$ is the identity map on $S$ since the two maps are the same on $f(R):$ $f(r)=f(r1_R)=f(r)s.$ Thus $s=1_S.$

The situation when $S$ is flat as left $R$-module is of special
interest. The following theorem characterizes such epimorphisms (for proof see  p. 227 in \cite{Stenstrom}).
\begin{theorem} For a ring homomorphism $f:R\rightarrow S$ the following conditions are equivalent.
\begin{enumerate}
\item $f$ is a ring epimorphism and makes $S$ into a flat left
$R$-module.

\item The family of right ideals $\ef=\{I | f(I)S=S\}$ is a
Gabriel filter and there is a ring isomorphism $g: S\cong R_{\ef}$ such that $g\circ f$ is the canonical map $q_R: R\rightarrow R_{\ef}.$

\item The following conditions are satisfied.
\begin{itemize}
\item[i)] For every $s\in S,$ there are $r_i\in R$ and $s_i\in S,$ $i=1,\ldots,n$ such that $sf(r_i)\in f(R)$ for all $i=1,\ldots,n$ and $\sum_{i=1}^n f(r_i)s_i=1.$

\item[ii)] If $f(r)=0,$ then there are $r_i\in R$ and $s_i\in S,$ $i=1,\ldots,n$ such that $r r_i=0$ for all $i=1,\ldots,n$ and $\sum_{i=1}^n f(r_i)s_i=1.$
\end{itemize}
\end{enumerate}
\label{PerfectQuotient}
\end{theorem}

If $f:R\rightarrow S$ satisfies the equivalent conditions of this
theorem, $S$ is called a {\em perfect right ring of quotients,} a flat epimorphic extension of $R,$ or
a perfect right localization of $R.$ A Gabriel filter $\ef$ and the corresponding torsion theory $\tau$ are
called {\em perfect} if the right ring of quotients $R_{\ef}$ is
perfect and $\ef=\{I| q_R(I)R_{\ef}=R_{\ef}\}$.

Note that for a filter $\ef,$ there is a
unique $R_{\ef}$-map $\Theta_M: M\otimes_R R_{\ef}\rightarrow
M_{\ef}$ with $q_M = \Theta_M i_M$ where $i_M$ is the canonical map $i_M: M\rightarrow
M\otimes_R R_{\ef}.$ Perfect filters are characterized by the property that the map $\Theta_M$ is an isomorphism
for every $M.$ Moreover, the following theorem holds (for proof see Thm. XI 3.4, p.
231 in \cite{Stenstrom}).
\begin{theorem} The following properties of a Gabriel filter $\ef$ are
equivalent.
\begin{enumerate}
\item $\ef$ is perfect.

\item The kernel of $i_M: M\rightarrow M\otimes_R R_{\ef}$ is the
torsion submodule of $M$ for the torsion theory determined by $\ef$ for every
module $M.$

\item $q_M(I)R_{\ef}=R_{\ef}$ for every $I\in \ef.$

\item The map $\Theta_M:M\otimes_R R_{\ef}\rightarrow M_{\ef}$ is an
isomorphism for every $M.$

\item The functor $q$ mapping the category of $R$-modules to the
category of $R_{\ef}$-modules given by $M\mapsto M_{\ef}$ is exact
and preserves direct sums.

\item $\ef$ has a basis consisting of finitely generated ideals
and the functor $q$ is exact.
\end{enumerate}
\label{perfect_filter}
\end{theorem}

Every ring has a maximal perfect right ring of quotients, unique up to
isomorphism (Theorem XI 4.1, p. 233, \cite{Stenstrom}). It is
called the {\em total right ring of quotients} (also maximal perfect
right localization of $R$, or maximal flat epimorphic right ring of
quotients of $R$) and is denoted by $\Qtot(R).$ $\Qtot(R)$ can be constructed as a directed union of all perfect right rings of quotients that are contained in $\Qmax(R).$

\section{Symmetric rings of quotients}
\label{section_symmetric_quotients}

If $R$ and $S$ are two rings, $\ef_l$ a Gabriel filter of left $R$-ideals and $\ef_r$ a Gabriel filters of right $S$-ideals, define the {\em filter induced by $\ef_l$ and $\ef_r$}, $_{\ef_l}\ef_{\ef_r}$ (or $_{l}\ef_{r}$ for short), to be  the set of right ideals of $S\otimes_{\Zset}R^{op}$ containing an ideal of the form $J\otimes R^{op}+S\otimes I$ where $I\in\ef_l$ and $J\in \ef_r.$ This defines a Gabriel filter (\cite{Ortega_thesis}, p. 100). We denote the corresponding {\em torsion theory induced by $\tau_l$ and $\tau_r$} by  $_l\tau_r.$
If $_RM_S$ is an $R$-$S$-bimodule, $\te_l(M),$ $\te_r(M)$ and $_l\te_r(M)$ torsion submodules of $M$ for $\tau_l,$ $\tau_r$ and  $_l\tau_r$ respectively, then $_l\te_r(M)=\te_l(M)\cap\te_r(M).$ Conversely, the class of all $R$-$S$-bimodules that are in $_l\te\cap \te_r$ defines the torsion class of the torsion theory induced by $_l\tau$ and $\tau_r.$

In \cite{Ortega_paper} and \cite{Ortega_thesis}, Ortega defines the {\em symmetric module of quotients} of $M$ with respect to $_l\ef_r$ to be
\[_{\ef_l}M_{\ef_r}=\dirlim_{K\in _l\ef_r}\;  \homo(K, \frac{M}{_l\te_r(M)})\]
where the homomorphisms in the formula are $S\otimes R^{op}$ homomorphisms (equivalently $R$-$S$-bimodule homomorphisms). Ortega shows that an equivalent approach can be obtained considering the compatible pairs of homomorphisms. If $M$ is an $R$-$S$-bimodule, $I\in\ef_l,$ $J\in\eg_r,$ $f:I\rightarrow\, _R\overline{M}$ an $R$-homomorphism and $g:J\rightarrow\overline{M}_S$ a $S$-homomorphism where $\overline{M}=M/_l\te_r(M),$ $(f,g)$ is a {\em compatible pair} if $f(i)j=ig(j)$ for all $i\in I,$ $j\in J.$ Define the equivalence relation by $(f,g)\sim(h,k)$ if and only if $f|_{I'}=h|_{I'}$ for some $I'\in \ef_l$ and $g|_{J'}=k|_{J'}$ for some $J'\in \eg_r.$ Then there is a bijective correspondence between elements of $_lM_r$ and the equivalence classes of compatible pairs of homomorphisms (Prop. 1.4 in \cite{Ortega_paper} or Prop. 4.37 in \cite{Ortega_thesis}).

As $_l M_r$ is the module of quotients with respect to filter of right ideals over a ring, we have the $R$-$S$-bimodule map $q_M: M\rightarrow\, _l M_r$ for every bimodule $_R M_S.$ The symmetric module of quotients has properties analogous to the right module of quotients. The lemma below lists these properties. They parallel those for one-sided version given in Lemma 1.2, Lemma 1.3, Lemma 1.5 on p. 196, Proposition 1.8 on p. 198 and observations on p. 197 and 199 in \cite{Stenstrom}).
\begin{lemma} (1) $_l\te_r(M)=\ker( q_M: M\rightarrow\, _l M_r)$ and  $_lM_r=\,_l(M/_l\te_r(M))_r$ is $_l\tau_r$-torsion-free.

(2) $M$ is torsion in $_l\tau_r$ if and only if $_l M_r=0.$

(3) Coker $q_M$ is $_l\tau_r$-torsion module. The maps $q_M$ for $R$-modules $M$ define a left exact functor $q$ from the category of $R$-$S$-bimodules onto itself.
\label{properties_of_module_of_quot}
\end{lemma}

If $R=S,$ the right-sided situation automatically ensures that $_l(R\otimes R^{op})_r$ has a ring structure. However, we would like to have a ring structure on $_lR_r.$  Fortunately, this is the case (see \cite{Schelter} and Lemma 1.5 in \cite{Ortega_paper}). The ring $_lR_r$ is called {\em the symmetric ring of quotients} with respect to torsion theory $_l\tau_r.$

\begin{example}
(1) Let $_l\de_r$ denote the filter induced by the filters of all dense right and left $R$-ideals $\de_r$ and $\de_l$. The corresponding symmetric ring of quotients is called the {\em maximal symmetric ring of quotients} and is denoted it by $\Qsimmax(R).$
Every symmetric ring of quotients of a symmetric torsion theory induced by hereditary and faithful left and right torsion theories embeds in $\Qsimmax(R).$

(2) If $R$ is a right and left Ore ring, the
classical left and the classical right rings of quotients coincide (see p. 303 in \cite{Lam}). We denote this ring by $\Qlrcl(R).$ The torsion theory for $R$-bimodules defined by the condition that an $R$-bimodule $M$ is a torsion module iff for every $m\in M$, there are nonzero regular elements $t,s\in R$ such that $sm=mt=0$ is induced by the torsion theories $\tau_{\Qlrcl(R)}$ and $_{\Qlrcl(R)}\tau$ (exercise 18, p. 318 in \cite{Lam}). This torsion theory is called the {\em classical torsion theory for an Ore ring}. $\Qlrcl(R)$ is equal to the symmetric ring of quotients with respect to the classical torsion theory (see Corollary \ref{Qsimtot_of_special_rings}).

(3) Let $f:R\rightarrow S$ be a ring homomorphism. The collection of all $R$-bimodules $M$ such that $S\otimes_R M\otimes_R S= 0$
is closed under quotients, extensions and direct sums. If $f$ makes $S$ into a flat right and a flat left $R$-module, then this
collection is closed under submodules and, hence, defines a hereditary torsion theory. We call this torsion theory {\em symmetric tensoring torsion theory} and denote it by $_S\tau_S.$ We denote its filter by  $_S\ef_S.$ Let $_S\ef$ and $\ef_S$ denote the filters of $_S\tau$ and $\tau_S$ respectively. Recall that $M$ is a $_S\tau$-torsion module if and only if $S\otimes_R M=0,$ and that that $M$ is a $\tau_S$-torsion module if and only if $M\otimes_R S=0.$
Also, $_S\ef=\{I$ left $R$-ideal $| Sf(I)=Sf(R) \},$ and $\ef_S=\{J$ right $R$-ideal $| f(J)S=f(R)S\}.$ If $f$ maps $1_R$ to $1_S,$  $_S\ef=\{I$ left $R$-ideal $| Sf(I)=S \},$ and $\ef_S=\{J$ right $R$-ideal $| f(J)S=S\}.$ We prove the following properties of the torsion theory $_S\tau_S.$
\label{Examples_symmetric}
\end{example}

\begin{proposition} Let $f:R\rightarrow S$ be a ring homomorphism that makes $S$ into a flat right and a flat left $R$-module and $M$ be an $R$-bimodule.
\begin{enumerate}
\item $_S\tau_S$ is torsion theory induced by torsion theories $_S\tau$ and $\tau_S$ and it coincides with the torsion theory $\tau_{S\otimes S^{op}}.$ Thus, $S\otimes_R M\otimes_R S=0$ if and only if $S\otimes_R M=0$ and $M\otimes_R S=0.$

\item $_S\te_S(M)=\;_S\te(M)\cap \te_S(M)=\ker(i_M: M\rightarrow S\otimes_R M\otimes_R S).$

\item If $f(1_R)=1_S,$ $_S\ef_S=\{K$ ideal of $R| Sf(K)S=S\}=\{K$ ideal of $R| f(K)(S\otimes S^{op})=S\otimes S^{op}\}=\{K$ ideal of $R|J\otimes R^{op}+R\otimes I\subseteq K$ for some left ideal $I$ and right ideal $J$ such that $Sf(I)=S$ and $f(J)S=S\}.$

\item If $f$ is a monomorphism, then all $R$-bimodules that are left and right $R$-flat are $_S\tau_S$-torsion-free. In particular, $_S\tau_S$ is faithful.

\item $S\otimes_R M\otimes_R S$ is $_S\tau_S$-torsion-free.
\end{enumerate}
\label{symmetric_tensor}
\end{proposition}
\begin{proof}
(1) + (2) To prove that $_S\tau_S$ is induced by $_S\tau$ and $\tau_S,$ it is sufficient to prove that $_S\te_S(M)=\;_S\te(M)\cap\te_S(M)$ for all $R$-bimodules $M.$ If $m\in M$ is in $_S\te(M)\cap\te_S(M),$ then $1\otimes m=0$ and $m\otimes 1=0.$ Then $1\otimes m\otimes 1=0$ so $m\in\;_S\te_S(M).$

To prove the converse, let us first note that $_S\te(M\otimes_R S)=\;_S\te(M)\otimes_R S.$ This can be seen from the following commutative diagram.
{\scriptsize
\begin{diagram}
0 & \rTo & _S\te(M\otimes_R S) && \rTo & M\otimes_R S && \rTo & S\otimes_R (M\otimes_R S)\\
& & \dTo_{=} && & \dTo_{=} & && \dTo_{\cong}\\
0 & \rTo & _S\te(M)\otimes_R S && \rTo & M\otimes_R S && \rTo & (S\otimes_R M)\otimes_R S
\end{diagram}
}

The bottom row is obtained by tensoring $0  \rightarrow\; _S\te(M)\rightarrow  M \rightarrow S\otimes_R M$ with the left $R$-flat module $S$ from the right. The top row is exact by part (2) of Lemma \ref{properties_of_onesided_tensor}. Identifying the modules $S\otimes_R (M\otimes_R S)$ and $(S\otimes_R M)\otimes_R S,$ we obtain the above commutative diagram. From the diagram, the first vertical map is the identity.

Similarly, we obtain the equality $\te_S(S\otimes_R M)=S\otimes_R \te_S(M).$

Applying the left exact functor $_S\te(\underline{\hskip.4cm}),$ to the exact sequence $0  \rightarrow\; \te_S(M)\rightarrow  M \rightarrow M\otimes_R S$ we obtain the exact sequences
$0  \rightarrow\; _S\te(\te_S(M))\rightarrow\; _S\te(M) \rightarrow\; _S\te(M\otimes_R S).$ This produces the following commutative diagram.
{\scriptsize
\begin{diagram}
0 & \rTo & _S\te(\te_S(M)) && \rTo & _S\te(M) && \rTo &  _S\te(M\otimes_R S)\\
& & \dTo_{=} && & \dTo_{=} && & \dTo_{=} \\
0 & \rTo & \te_S(\;_S\te(M)) && \rTo & _S\te(M) && \rTo & _S\te(M)\otimes_R S
\end{diagram}
}

The bottom row is exact by part (2) of Lemma \ref{properties_of_onesided_tensor}. The last vertical equality follows from the previous commutative diagram and the first vertical equality follows then from the fact that the last two vertical maps are equalities.

Consider now the following diagram.
{\scriptsize
\begin{diagram}
0 & & 0 & & 0 & & 0 & &  & &\\
&\rdTo&\dTo & & \dTo & & \dTo &  & &  & \\
0 & \rTo & _S\te(\te_S(M)) & \rTo & _S\te(M) & \rTo & _S\te(M)\otimes_R S & \rTo & _S\te(M)\otimes_R S/f(R) & \rTo & 0\\
&&\dTo& \rdTo& \dTo & & \dTo &  &  &  & \\
0 & \rTo & \te_S(M) & \rTo & M & \rTo & M\otimes_R S & \rTo & M\otimes_R S/f(R) & \rTo & 0\\
&&\dTo& & \dTo & \rdTo & \dTo&  & \dTo &  & \\
0 & \rTo & S\otimes_R \te_S(M) & \rTo & S\otimes_R M & \rTo & S\otimes_R M\otimes_R S & \rTo & S\otimes_R M\otimes_R S/f(R) & \rTo & 0\\
&&\dTo & & \dTo & & \dTo &  & \dTo &  & \\
& & S/f(R)\otimes_R \te_S(M) &  & S/f(R)\otimes_R M & \rTo & S/f(R)\otimes_R M\otimes_R S & \rTo & S/f(R)\otimes_R M\otimes_R S/f(R) & \rTo & 0\\
&&\dTo & & \dTo & & \dTo & & \dTo & & \\
&&0 & & 0 & & 0 & & 0 & & \\
\end{diagram}
}

The first two nonzero columns and rows are exact by part (2) of Lemma \ref{properties_of_onesided_tensor}. The third nonzero column and row are exact by left and right flatness of $S$.
The last nonzero row and column are exact because the functors $S/f(R)\otimes_R\underline{\hskip.4cm}$ and $\underline{\hskip.4cm}\otimes_R S/f(R)$ are right exact.

The top left square is commutative since $_S\te(\te_S(M))=\te_S(\;_S\te(M)).$ Thus, we have that all horizontal and vertical maps commute.

Now let us consider the diagonal maps. The map $_S\te(\te_S(M))\rightarrow M$ is injective as it is an inclusion. The image of this map is contained in the kernel of $i_M: M\rightarrow S\otimes_R M\otimes_R S$ as $_S\te(\te_S(M))\subseteq\; _S\te(M)\cap \te_S(M)\subseteq \ker i_M= \,_S\te_S(M).$ To prove the converse, let $m\in\ker i_M.$ Consider $m\otimes 1\in M\otimes_R S .$ As $m\otimes 1$ maps to 0 by the map $M\otimes_R S\rightarrow S\otimes_R M\otimes_R S,$ $m\otimes_R 1$ is in $_S\te(M)\otimes_R S.$ $m\otimes 1$ maps to 0 under map $_S\te(M)\otimes_R S \rightarrow\; _S\te(M)\otimes_R S/R$ and so it is in image of $_S\te(M) \rightarrow\; _S\te(M)\otimes_R S.$ As the vertical map $_S\te(M)\rightarrow M$ is inclusion and the maps in the second square in the first row commute, $m$ is in $_S\te(M).$ But then, $1\otimes m=0.$ Symmetrically, we obtain that $m\otimes 1=0.$ This proves the exactness of the diagonal maps and that $_S\te(\te_S(M))=\, _S\te_S(M).$ As
we showed that $_S\te(M)\cap\te_S(M)\subseteq\; _S\te_S(M)$ and $_S\te(\te_S(M))\subseteq\; _S\te(M)\cap\te_S(M),$ we obtain that $_S\te(M)\cap\te_S(M)=\,_S\te_S(M).$

If we identify the right action of $R\otimes R^{op}$ with the $R$-bimodule action, we can identify $S\otimes_R M\otimes_R S$ with $M\otimes_{R\otimes R^{op}}(S\otimes S^{op}).$ Then $_S\te_S(M)=\ker(i_M: M\rightarrow S\otimes_R M\otimes_R S)=\ker(M\rightarrow M\otimes_{R\otimes R^{op}}(S\otimes S^{op}))=\te_{S\otimes S^{op}}(M)$ so $_S\tau_S =\tau_{S\otimes S^{op}}.$ This proves parts (1) and  (2) of the proposition.

(3) To prove the first two equalities, note that an ideal $K$ of $R$ is in $_S\ef_S$ if and only if $R/K$ is in $_S\te_S.$ By parts (1) and (2), this is equivalent with $S\otimes_R R/K \otimes_R S=0$ which, in turn, is equivalent to $R/K\otimes_{R\otimes R^{op}}(S\otimes S^{op})=0.$ Since $S$ is left and right $R$-flat, $S\otimes S^{op}$ is flat as left $R\otimes R^{op}$-module. Thus $R/K\otimes_{R\otimes R^{op}}(S\otimes S^{op})=0$ is equivalent to $K\otimes_{R\otimes R^{op}}(S\otimes S^{op})=S\otimes S^{op}$ which, in turn, is equivalent to $f(K)(S\otimes S^{op})=S\otimes S^{op}.$ But this last condition is equivalent to $Sf(K)S=S.$ Thus, $K\in\;_S\ef_S$ iff $S=Sf(K)S$ iff $f(K)(S\otimes S^{op})=S\otimes S^{op}.$ The last set in the claim describes the filter induced with $\ef_S$ and $_S\ef.$ But as $_S\tau_S$ is the torsion theory induced by $\tau_S$ and $_S\tau,$ the last equality follows.

(4) If $f$ is a monomorphism and $M$ is an $R$-bimodule that is left and right flat, then $_S\te(M)=\te_S(M)=0$ by part (3) of Lemma \ref{properties_of_onesided_tensor}. Thus, $_S\te_S(M)=\te_S(M)\cap\;_S\te(M)=0.$ In particular, $R$ is $_S\tau_S$-torsion-free.

(5) Tensoring the exact sequence $0\rightarrow\, _S\te_S(M)\rightarrow M\rightarrow S\otimes_R M\otimes_R S$ with $S$ from both sides, we obtain $0\rightarrow0\rightarrow S\otimes_R M\otimes_R S\rightarrow S\otimes_R(S\otimes_R M\otimes_R S)\otimes_R S.$ Thus, $S\otimes_R M\otimes_R S$ embeds in $S\otimes_R(S\otimes_R M\otimes_R S)\otimes_R S$ and so it is torsion-free.
\end{proof}

For an $R$-bimodule $M,$ it readily follows that $_lM_r$ is a right $_l(R\otimes R^{op})_r$-module.
Unfortunately, some additional conditions on $M$ need to be assumed in order to obtain the $_lR_r$-bimodule (right $_lR_r\otimes\,_lR_r^{op}$) structure. If $R$ and $S$ are two rings, $\ef_l$ and $\ef_r$ Gabriel filters of right and left ideals for $R,$  $\eg_l$ and $\eg_r$ Gabriel filters of right and left ideals for $S,$ and $M$ an $R$-$S$-bimodule, Ortega shows that   $_{\ef_l}M_{\eg_r}$ is a $_{\ef_l}R_{\ef_r}$-$_{\eg_l}S_{\eg_r}$-bimodule if the following fours conditions hold (see Lemmas 2.1 and 2.2 in \cite{Ortega_paper}  or Lemmas 4.47 and 4.48 in \cite{Ortega_thesis} for details).
\begin{itemize}
\item[Q1] $M$ is a $\eg_r$-torsion-free right $R$-module.

\item[Q2] $M$ is a $\ef_l$-torsion-free left $R$-module.

\item[Q3] For every $J\in\ef_r,$ $M/JM$ is $\eg_r$-torsion right $R$-module.

\item[Q4] For every $I\in\eg_l,$ $M/MI$ is $\ef_l$-torsion left $R$-module.
\end{itemize}

The following proposition shows that the bimodules satisfying the above four conditions have a symmetric version of the map $\Theta_M.$
\begin{proposition}
If $_lR_r$ is a symmetric ring of quotients with respect to the filter induced by a left filter $\ef_l$ and a right filter $\ef_r,$ and $M$ an $R$-bimodule that satisfies $Q1$--$Q4$, then there is a commutative diagram
{\scriptsize
\begin{diagram}[tight]
M & \rTo^{i_M} & & _{\ef_l}R_{\ef_r}\otimes_R M\otimes_R\; _{\ef_l}R_{\ef_r} \\
\dTo^{q_M} &    \ldTo_{\Theta_M}& & \\
_{\ef_l}M_{\ef_r}& & & \\
\end{diagram}
}
where the horizontal map is the natural $R$-bimodule map $i_M,$ the vertical map is the $R$-bimodule map $q_M,$ and the  diagonal map $\Theta_M$ is the $_lR_r$-bimodule map defined by $\Theta_M(s\otimes m\otimes t)=sq_M(m)t.$ $\Theta_M$ is unique $_lR_r$-bimodule map with $q_M=\Theta_M i_M.$
\label{Map_Theta}
\end{proposition}
\begin{proof}
As $q_M$ is an $R$-bimodule map (see paragraph preceding Lemma \ref{properties_of_module_of_quot}), the mapping
$_lR_r\times M\times\; _lR_r\rightarrow \, _lM_r$ given by $(s, m, t)\mapsto s q_M(m) t$ is middle tri-linear. The term $s q_M(m) t$ is well defined if $M$ satisfies $Q1$-$Q4.$ Thus, this mapping defines a $_lR_r$-bimodule mapping $\Theta_M:\,_lR_r\otimes_R M\otimes_R\, _lR_r\rightarrow\,_lM_r$ such that $\Theta_M: s\otimes m\otimes t\mapsto s q_M(m) t.$ In this case $q_M=\Theta_M\circ i_M$ by definition and it is easy to see that $\Theta_M$ is unique $_lR_r$-bimodule map with this property.
\end{proof}

\section{Perfect symmetric ring of quotients}
\label{section_symmetric_perfect}

In this section, we prove the symmetric versions of Theorems \ref{PerfectQuotient} and \ref{perfect_filter}.

\begin{theorem} For a ring homomorphism $f:R\rightarrow S$ the following conditions are equivalent.
\begin{enumerate}
\item $f$ is a ring epimorphism and makes $S$ into a flat left and a flat right
$R$-module.

\item The family of left ideals $\ef_l=\{I | Sf(I)=S\}$ is a
left Gabriel filter, the family of right ideals $\ef_r=\{J | f(J)S=S\}$ is a right
Gabriel filter and there is a ring isomorphism $g: S\cong\; _{\ef_l}R_{\ef_r}$ such that
$g\circ f$ is the canonical map $q_R: R\rightarrow\; _{\ef_l}R_{\ef_r}.$

\item The following conditions are satisfied.
\begin{itemize}
\item[i)] For every $s\in S,$ there are $r_i\in R$ and $s_i\in S,$ $i=1,\ldots,n$ and
$r_j\in R$ and $s_j\in S,$ $j=1,\ldots,m$
such that $f(r_i)s\in f(R)$ for all $i,$ $sf(r_j)\in f(R)$ for all $j,$ $\sum_{i=1}^n s_if(r_i)=1,$ and $\sum_{j=1}^m  f(r_j)s_j=1.$

\item[ii)] If $f(r)=0,$ then there are $r_i\in R$ and $s_i\in S,$ $i=1,\ldots,n$ and $r_j\in R$ and $s_j\in S,$ $j=1,\ldots,m$ such that $r_ir=0$ for all $i,$ $rr_j=0$ for all $j,$ $\sum_{i=1}^n s_if(r_i)=1,$  and $\sum_{j=1}^m f(r_j)s_j=1.$
\end{itemize}
\end{enumerate}
\label{PerfectSymmetricQuotient}
\end{theorem}

\begin{proof}
(1) $\Rightarrow$ (2) If (1) holds, then $S$ is also a perfect left ring of quotients with respect to perfect Gabriel filter $\ef_l$ by Theorem \ref{PerfectQuotient}. Similarly, $\ef_r$ is a right Gabriel filter. The filter of two-sided ideals of $R$ induced by $\ef_l$ and $\ef_r$ is the set of $R$-ideals such that $Sf(K)S=S$ by part (3) of Proposition \ref{symmetric_tensor}.

Let $M$ be any $R$-bimodule. We show that
\begin{equation}
_lM_r  \cong\, _l(S\otimes_R M\otimes_R S)_r
\end{equation}
To prove this isomorphism, consider the following exact sequence.
\begin{diagram}[tight]
0&\rTo&_S\te_S(M)&\rTo& M&\rTo &S\otimes_R M\otimes_R S&\rTo&& \coker(i_M)&\rTo& 0
\end{diagram}
Note that the middle map $i_M$ is given by $M\cong R\otimes_R M\otimes_R R\stackrel{f\otimes 1\otimes f}{\longrightarrow}
S\otimes_R M\otimes_R S.$ Tensoring the exact sequence from both sides with $S,$ we obtain
\[0\rightarrow S\otimes_R M\otimes_R S\rightarrow S\otimes_RS\otimes_R M\otimes_R S\otimes_R S\rightarrow S\otimes_R\coker(i_M)\otimes_R S\rightarrow 0.\] The first nonzero map is bijective as $f$ is epimorphism and so
$S\otimes_R\coker(i_M)\otimes_R S=0$. Thus, $\coker(i_M)$ is a $_S\tau_S$-torsion module and so $_l(\coker(i_M))_r=0.$

As $_l(\coker(i_M))_r=0$ and $_l(i_M(M))_r=\,_l(M/_S\te_S(M))_r=\,_l(M)_r$ by part (1) of Lemma \ref{properties_of_module_of_quot},
applying the left exact functor $q$ on the exact sequence $0\rightarrow\; _S\te_S(M)\rightarrow M \rightarrow S\otimes_R M\otimes_R S\rightarrow\coker(i_M)\rightarrow0$ produces an isomorphism $q(i_M):\,_l(M)_r\cong\,_l(S\otimes_R M\otimes_R S)_r.$

The isomorphism (4.1)
asserts that $_lR_r\cong\,_l(S\otimes_R R\otimes_R S)_r\cong\,_l(S\otimes_R S\otimes_R S)_r\cong\,_lS_r.$ The middle map is an isomorphism as $f$ is a ring epimorphism. Thus $_lR_r\cong\,_lS_r$ via $q(f)$ since the composition $R\rightarrow S\otimes_R R\otimes_R S\rightarrow S\otimes_R S\otimes_R S\rightarrow S$ is $f.$

As $S\cong S\otimes_R S\otimes_R S$, the $_S\tau_S$-torsion submodule of $S$ is 0 so $S$ is $_S\tau_S$-torsion-free. Thus $q_S: S\rightarrow\,_lS_r$ is an embedding. We shall show that it is onto. Since $S$ is right perfect, the proof of the right-sided version of the theorem (\cite{Stenstrom}, top of p. 228) establishes that every right $R$-homomorphisms $\beta: J\rightarrow S,$ $J\in\ef_r$ can be extended to $R.$ Thus, $\beta$ is the left multiplication $L_t$ by some $t\in S.$ Similarly, every left $R$-homomorphism $\alpha:I\rightarrow S,$ $I\in\ef_l$ can be extended to $R.$ Thus, $\alpha$ is the right multiplication $R_s$ by some $s\in S.$ If $\alpha$ and $\beta$ are compatible, $isj=R_s(i)j=\alpha(i)j=i\beta(j)=iL_t(j)=itj$ for every $i\in I$ and $j\in J.$ Thus, the left multiplication by $is-it$ is a right $R$-map that maps $J$ onto 0 and so it factors to a map $R/J\rightarrow S.$ But as $S$ is $\tau_S$-torsion-free ($S$ embeds into $S\otimes_R S$), there are no nonzero maps from $\tau_S$-torsion modules into $S.$ Hence, $is-it=0.$ As this holds for every $i\in I,$ the right multiplication by $s-t$ is a left $R$-map that maps $I$ onto 0 and so it factors to a map $R/I\rightarrow S.$ But as $S$ is $_S\tau$-torsion-free, it is zero on entire $R$. Thus $s=t.$

This asserts that every pair of compatible homomorphisms $(\alpha, \beta)$ from $_lS_r$ is of the form $(R_s, L_s)$ for some $s\in S.$ Because of this, the map $q_S$ is onto and so it is an isomorphism.

Thus, we obtain the diagram below with $q_S$ and $q(f)$ being isomorphisms.
{\scriptsize
\begin{diagram}
R & \rTo^{f} & S\\
\dTo^{q_R} & & \dTo_{q_S}\\
_lR_r & \rTo^{q(f)} & _lS_r
\end{diagram}
}

Defining $g=q(f)^{-1}\circ q_S$ we obtained the desired isomorphism.

(2) $\Rightarrow$ (3) For $s\in S\cong\, _lR_r,$ there are $I\in\ef_l$ and $J\in \ef_r$ such that $f(I)s\subseteq f(R)$ and $sf(J)\subseteq f(R).$ As $Sf(I)=S$ and $f(J)S=S,$ $1=\sum_{i=1}^n s_if(r_i)=\sum_{j=1}^m f(r_j)s_j$ for some $r_i\in R$ and $s_i\in S,$ $i=1,\ldots,n$ and $r_j\in R$ and $s_j\in S,$ $j=1,\ldots,m.$ As $f(I)s\subseteq f(R)$ and $sf(J)\subseteq f(R),$ $f(r_i)s\in f(R)$ for all $i,$ $sf(r_j)\in f(R)$ for all $j.$

If $f(r)=0,$ then $r\in\, _l\te_r(R).$ Thus, there are $I\in\ef_l$ and $J\in \ef_r$ such that $Ir=0$ and $rJ=0.$ As
$Sf(I)=S$ and $f(J)S=S,$ again $1=\sum_{i=1}^n s_if(r_i)=\sum_{j=1}^m f(r_j)s_j$ for some $r_i\in R$ and $s_i\in S,$ $i=1,\ldots,n$ and $r_j\in R$ and $s_j\in S,$ $j=1,\ldots,m.$ As $Ir=0$ and $rJ=0,$ $r_ir=0$ for all $i,$ and $rr_j=0$ for all $j.$

(3) $\Rightarrow$ (1)
By Lemma 2.3, p. 228 from \cite{Stenstrom}, the right-sided part of condition i) implies $S/f(R)\otimes_R S=0.$ The left-sided part of condition i) ensures that $S\otimes_R S/f(R)=0$ as well. In this case we have that $f$ is a ring epimorphism (see Proposition XI 1.2, p. 226 in \cite{Stenstrom}).

The right-sided part of condition ii) is exactly the criterion for left flatness (see Equational Criteria for Flatness, Theorem 4.24, p. 130 in \cite{Lam}, or \cite{Stenstrom} bottom of p. 228). Thus, $f$ makes $S$ into a left flat $R$-module. Symmetrically, the left-sided part of condition ii) implies that $f$ makes $S$ into a right flat $R$-module.
\end{proof}

\begin{definition}
If $f:R\rightarrow S$ satisfies the equivalent conditions of this
theorem, $S$ is called a {\em perfect symmetric ring of quotients} or a {\em perfect symmetric localization of $R$.}
A hereditary torsion theory $_l\tau_r$ with filter $_l\ef_r$ induced by $\ef_l$ and $\ef_r$ is {\em perfect} if the symmetric ring of quotients $_lR_r$ is
perfect, $\ef_l=\{I| _lR_r q_R(I)=\;_lR_r\}$ and $\ef_r=\{J| q_R(J)\;_lR_r=\;_lR_r\}.$ The Gabriel filter $_l\ef_r$ is called {\em perfect} in this case.
\end{definition}

\begin{lemma} If $f:R\rightarrow S$ is a ring epi that makes $S$ into a perfect symmetric ring of quotients, then there is a ring isomorphism $j: S\cong\,_{\ef_l}R$ where  $\ef_l$ is the filter of left ideals $\{I| Sf(I)=S\}$ and a ring isomorphism $k: S\cong R_{\ef_r}$ where  $\ef_r$ is the filter of right ideals $\{J| f(J)S=S\}.$ $\ef_l$ is a perfect left and $\ef_r$ is a perfect right filter.  The isomorphisms $j$ and $k$ are such that $jf$ is the canonical map $R\rightarrow\,_{\ef_l}R$ and $kf$ is the canonical map $R\rightarrow R_{\ef_r}.$
\label{lemma_perfect}
\end{lemma}
\begin{proof}
If $S$ is a perfect ring of quotients, then condition (1) of Theorem \ref{PerfectSymmetricQuotient} implies the condition (1) of Theorem \ref{PerfectQuotient} so $S$ is also a right perfect ring of quotients. Similarly, $S$ is also a perfect left ring of quotients. Condition (2) of Theorem \ref{PerfectSymmetricQuotient} and condition (1) of Theorem \ref{PerfectQuotient} imply that
the set of right ideals $\{J| f(J)S=S\}$ is a Gabriel filter $\ef_l$ and that $S$ is isomorphic to $R_{\ef_r}.$ Similarly, $S$ is isomorphic to $_{\ef_l}R$ for $\ef_l=\{I| Sf(I)=S\}.$ The last sentence follows also from condition (2) of Theorem \ref{PerfectQuotient}.
\end{proof}

Next, we prove the symmetric version of Theorem \ref{perfect_filter} characterizing the perfect symmetric filters. First we introduce some notation. If $\ef_l$ is a left and $\ef_r$ a right Gabriel filter, $_l\ef_r$ the symmetric filter induced by $\ef_l$ and $\ef_r,$ and $_lR_r$ the corresponding symmetric quotient, let $q^l$ denote the functor of left $R$-modules defined by $q^l_M:$ $M\rightarrow\, _{\ef_l}M$ and $q^r$ denote the functor of right $R$-modules defined by $q^r_M:$ $M\rightarrow M_{\ef_r}.$
Moreover, let $i_M^l$ denote the functor of left $R$-modules defined by $M\rightarrow\;_lR_r \otimes_R M$ and $i_M^r$ denote the functor of right $R$-modules defined by $M\rightarrow M\otimes_R\,_lR_r.$ As in Proposition \ref{symmetric_tensor}, $i_M$ denotes the map  $i_M: M\rightarrow\,_lR_r \otimes_R M\otimes_R\, _lR_r$ and $q_M$ denotes the map $M\rightarrow\, _lM_r$ if $M$ is an $R$-bimodule.

\begin{theorem} The following properties of a Gabriel filter $_l\ef_r$ induced by $\ef_l$ and $\ef_r$ are equivalent.
\begin{itemize}
\item[(1)] $_l\ef_r$ is perfect.

\item[(2)] $\te_l(K)=\ker i_K^l$ for every left $R$-module $K$, $\te_r(L)=\ker i_L^r$ for every right $R$-module $L,$ and $_l\te_r(M)=\ker i_M$ for every $R$-bimodule $M.$

\item[(2')] $_lR_r\otimes_R K=0$ for every $\tau_l$-torsion module $K,$ $L\otimes_R\, _lR_r=0$ for every $\tau_r$-torsion module $L,$ and $_lR_r\otimes_R M\otimes_R\, _lR_r=0$ for every $_l\tau_r$-torsion module $M.$

\item[(3)] $_lR_r q_R(I)=\,_lR_r$ for every $I\in \ef_l,$  $q_R(J)_lR_r=\, _lR_r$ for every $J\in \ef_r,$ and $_lR_r q_R(K)\,_lR_r=\,_lR_r$ for every $K\in\,_l\ef_r.$

\item[(4)] There is a ring isomorphism $j:\,_lR_r\rightarrow\,_{\ef_l}R$ such that $jq_R=q^l_R,$ and a ring isomorphism $k:\,_lR_r\rightarrow R_{\ef_r}$ such that $kq_R=q^r_R.$ For every left $R$-module $K$ there is a left $_lR_r$-module isomorphism $\Phi^l_K: \,_lR_r \otimes_R K\cong \;_{\ef_l}K$ such that $\Phi^l_K i^l_K=q^l_K.$ For every right $R$-module $L$ there is a right $_lR_r$-module isomorphism $\Phi^r_L: L\otimes_R \,_lR_r \cong L_{\ef_r}$ such that $\Phi^r_L i^r_L=q^r_L.$ For every $R$-bimodule $M,$ there is an $R$-bimodule isomorphism \[\Phi_M: \,_lR_r \otimes_R M\otimes_R\, _lR_r\cong \;_lM_r\] such that $\Phi_M i_M=q_M.$ If $M$ satisfies $Q1$-$Q4,$ then $\Phi_M=\Theta_M$ (see Proposition \ref{Map_Theta}) and it is a $_lR_r$-bimodule isomorphism.

\item[(5)] There is a ring isomorphism $j:\,_lR_r\rightarrow\,_{\ef_l}R$ such that $jq_R=q^l_R,$ and a ring isomorphism $k:\,_lR_r\rightarrow R_{\ef_r}$ such that $kq_R=q^r_R.$ Functors $q^l,$ $q^r$ and $q$ are exact and preserve direct sums.

\item[(6)] There is a ring isomorphism $j:\,_lR_r\rightarrow\,_{\ef_l}R$ such that $jq_R=q^l_R,$ and a ring isomorphism $k:\,_lR_r\rightarrow R_{\ef_r}$ such that $kq_R=q^r_R.$ Functors $q^l,$  $q^r, $ and $q$ are exact and $\ef_l,$ $\ef_r,$ and $_l\ef_r$ have basis consisting of finitely generated ideals (left, right, and both sided respectively).
\end{itemize}
\label{perfect_symmetric_filter}
\end{theorem}
\begin{proof}
$(1)\Rightarrow(2)$ If $_l\ef_r$ is perfect, then $\ef_l=\{I| _lR_r q_R(I)=\;_lR_r\}$ and $\ef_r=\{J| q_R(J)\;_lR_r=\;_lR_r\}.$
Thus, the condition (2) of Theorem \ref{PerfectSymmetricQuotient} is fulfilled and so (1) of Theorem \ref{PerfectSymmetricQuotient} holds as well. So, $_lR_r$ is flat both as left and right $R$-module.
The theory $\tau_l$ is equal to the left tensoring torsion theory $\,_{_lR_r}\tau$ and $\tau_r$ is equal to the right tensoring torsion theory $\tau_{_lR_r}$ and so $\te_l(K)=\ker(K\rightarrow\;_lR_r \otimes_R K)$ for every left $R$-module $K$, $\te_r(L)=\ker(L\rightarrow L\otimes_R\; _lR_r)$ for every right $R$-module $L.$ From part (3) of Proposition \ref{symmetric_tensor} it follows that $_l\tau_r$ is equal to the symmetric tensoring torsion theory $_{_lR_r}\tau_{_lR_r}$ and that $_l\te_r(M)=\,_{_lR_r}\te_{_lR_r}(M)=\, _{_lR_r}\te(M)\cap \te_{_lR_r}(M)=\ker(i_M).$

$(2)\Rightarrow(2')$ If $K$ is in $\te_l,$ $K=\te_l(K)=\ker(K\rightarrow\;_lR_r \otimes_R K).$ Thus, $1\otimes k=0$ for every $k\in K.$ But then $q\otimes k = q(1\otimes k)=0$ for every $q\in\,_lR_r$ and so $_lR_r\otimes_R K=0.$
The statements for left and symmetric modules follow similarly.

$(2')\Rightarrow (3)$ Let $I\in \ef_l.$ Then $R/I$ is $\te_l$-torsion module. By assumption $0=\,_lR_r\otimes_R R/I=\,_lR_r/\,_lR_rq_R(I).$ Thus $\,_lR_r=\,_lR_rq_R(I).$ The statement for right ideals is proven similarly.

Now let $K\in\,_l\ef_r.$ For $K,$ there is $I\in \ef_l$ and $J\in \ef_r$ such that $IR\subseteq K$ and $RJ\subseteq K.$ Thus, $q_R(I)Rq_R(J)\subseteq q_R(K).$ As $_lR_r q_R(I)=\,_lR_r$ and $q_R(J)\,_lR_r=\,_lR_r,$
\[_lR_r=\,_lR_r\,_lR_r=\,_lR_r q_R(I)q_R(J)\,_lR_r\subseteq\, _lR_rq_R(I)Rq_R(J)\,_lR_r\subseteq\, _lR_rq_R(K)\,_lR_r.\]
Thus $_lR_r=\,_lR_rq_R(K)\,_lR_r.$

$(3)\Rightarrow(1)$ By Theorem \ref{PerfectSymmetricQuotient}, it is sufficient to show that $\ef_l=\{I|_lR_rq_R(I)=\,_lR_r\}$ and $\ef_r=\{J| q_R(J)\,_lR_r=\,_lR_r\}.$ By (3), $\ef_l\subseteq\{I|_lR_rq_R(I)=\,_lR_r\}$ and $\ef_r\subseteq \{J| q_R(J)\,_lR_r=\,_lR_r\}.$ Thus, we need to show that $I\in\ef_l$ for every left ideal with $_lR_rq_R(I)=\,_lR_r,$ and that $J\in\ef_r$ for every right ideal with $q_R(J)\,_lR_r=\,_lR_r.$

Let $I$ be a left $R$-ideal such that $_lR_rq_R(I)=\,_lR_r.$ Then $1=\sum_{i=1}^n q_i q_R(r_i)$ for some $r_i\in I$ and $q_i\in\, _lR_r.$ For every such $q_i,$ there exist a left ideal $I_i\in \ef_l$ such that $q_R(I_i)q_i\subseteq q_R(R)$ and a right ideal $J_i$ such that $q_i q_R(J_i)\subseteq q_R(R).$ Let $I'=\bigcap I_i.$ Then $I'\in \ef_l$ and $q_R(I')q_i\subseteq q_R(R)$ for every $i.$ Thus, $q_R(I')=q_R(I')1=\sum q_R(I')q_i q_R(r_i) \subseteq \sum q_R(R)q_R(r_i)\subseteq q_R(I).$ As $q_R(I')\subseteq q_R(I),$ the map $q_{R/I'}(R/I')\rightarrow q_{R/I}(R/I)$ is well defined and onto. $I'\in \ef_l$ implies  $q_{R/I'}(R/I')=0.$ But then  $q_{R/I}(R/I)=0$ as well and so $R/I\in\te_l,$ i.e. $I\in \ef_l.$ Similarly, we prove the claim for right ideals.

This finishes the proof that $(1)\Leftrightarrow(2)\Leftrightarrow(2')\Leftrightarrow(3).$ Note that Lemma \ref{lemma_perfect} proves that (1) implies that $_lR_r$ is isomorphic (as a ring) to both $_{\ef_l}R$ and $R_{\ef_r}.$

$(1)+(2')\Rightarrow (4)$ Condition (1) implies the first sentence of (4) by Lemma \ref{lemma_perfect}.
The existence of isomorphisms $\Phi^l$ and $\Phi^r$ then follows directly from Theorem \ref{perfect_filter}.

From the proof of Theorem \ref{PerfectSymmetricQuotient}, we have that the $R$-module map $q(i_M):\,_l(M)_r\cong\,_l(\,_lR_r\otimes_R M\otimes_R\, _lR_r)_r$ is an isomorphism if $_l\ef_r$ is perfect. Consider now the exact sequence $0\rightarrow\,_l\te_r(M)\rightarrow M\rightarrow\,_lM_r \rightarrow\coker(q_M)\rightarrow0.$ Tensor it with $_lR_r$ from both sides. As $_l\te_r(M)$ and $\coker(q_M)$ are $_l\tau_r$-torsion, they are $_{_lR_r}\tau_{_lR_r}$-torsion as well by part (2'). So, we have that $i(q_M)$ is the isomorphism
\[_lR_r\otimes_R M\otimes_R\, _lR_r\cong\, _lR_r\otimes_R \,_lM_r\otimes_R\, _lR_r.\]

Consider the following $R$-bimodule maps \[\Phi_M:\,_lR_r\otimes_R M\otimes_R\, _lR_r\rightarrow\,_l(\,_lR_r\otimes_R M\otimes_R\, _lR_r)_r\cong\,_lM_r\] where the first map is $q_{_lR_r\otimes_R M\otimes_R\, _lR_r}$ and the second is the inverse of the isomorphism $q(i_M),$ and \[\Psi_M:\,_lM_r\rightarrow\,_lR_r\otimes_R\,_lM_r\otimes_R\, _lR_r\cong \,_lR_r\otimes_R M\otimes_R\, _lR_r\] where the first map is $i_{_lM_r}$ and the second the inverse of the isomorphism $i(q_M).$  If $m\in M,$ $q_M(m)$ can be represented by a compatible pair of homomorphism $(R_m, L_m)$ where $R_m$ is the right multiplication by $m$ and $L_m$ is the left multiplication by $m.$ The map $\Phi_M,$ maps $i_M(m)=1\otimes m\otimes 1$ to $(R_m, L_m)$ and $\Psi_M$ maps $(R_m,L_m)$ to $1\otimes m\otimes 1.$ So, we have that $\Phi_M i_M=q_M$ and $\Psi_M q_M=i_M.$

Thus, $\Phi_M\Psi_M$ is the identity when restricted to $q_M(M)$ and $1_{_lM_r}-\Phi_M\Psi_M$ factors to an $R$-map $_lM_r/q_M(M)\rightarrow\, _lM_r.$ As the module $_lM_r/q_M(M)$ is $_l\tau_r$-torsion (part (3) of Proposition \ref{properties_of_module_of_quot})
and the module $_lM_r$ is $_l\tau_r$-torsion-free, $1_{_lM_r}-\Phi_M\Psi_M=0$ on entire $_lM_r$ and so $\Phi_M\Psi_M=1_{_lM_r}.$

Also, the map $\Psi_M\Phi_M$ is the identity when restricted to $i_M(M).$ The module $_lR_r\otimes_R M\otimes_R\, _lR_r$ is $_{_lR_r}\tau_{_lR_r}$-torsion-free (part (5) of Proposition \ref{symmetric_tensor}) and the quotient $\coker(i_M)=(_lR_r\otimes_R M\otimes_R\, _lR_r)/i_M(M)$ is $_{_lR_r}\tau_{_lR_r}$-torsion module (see proof of $(1)\Rightarrow(2)$ in Theorem \ref{PerfectSymmetricQuotient}). Thus, we obtain that $\Psi_M\Phi_M$ is the identity on entire $_lR_r\otimes_R M\otimes_R\, _lR_r.$

If $M$ satisfies $Q1$-$Q4,$ then $\Phi_M=\Theta_M$ on $i_M(M).$ Then $\Phi_M-\Theta_M$ factors to a map $(\,_lR_r\otimes_R M\otimes_R\, _lR_r)/i_M(M)\rightarrow\,_lM_r.$ The first module is $_{_lR_r}\tau_{_lR_r}$-torsion and so $_l\tau_r$-torsion also by part (2'). The second module is $_l\tau_r$-torsion-free. Thus, this map is 0 and so $\Phi_M=\Theta_M$ on entire $_lR_r\otimes_R M\otimes_R\, _lR_r.$ Thus, $\Phi_M$ is a $_lR_r$-bimodule map and $\Theta_M$ is an isomorphism because $\Phi_M$ is.

$(4)\Rightarrow(2)$ From $\Phi^l_K i^l_K=q^l_K,$ it follows that $\te_l(K)=\ker q^l_K=\ker i^l_K$ for every left $R$-module $K.$ Similarly $\Phi^r_L i^r_L=q^r_L,$ implies that $\te_r(L)=\ker q^r_L=\ker i^r_L$ for every right $R$-module $L$ and $\Phi_Mi_M=q_M$ implies that $_l\te_r(M)=\ker q_M=\ker i_M$ for every $M.$

This proves that $(1)\Leftrightarrow(2)\Leftrightarrow(2')\Leftrightarrow(3)\Leftrightarrow(4).$ For the remaining part we shall show that $(5)\Rightarrow(3),$ $(1)+(4)\Rightarrow(5),$ $(6)\Rightarrow(3),$ and $(1)+(4)\Rightarrow(6).$

$(5)\Rightarrow(3)$ If $q^l$ is exact and preserves direct sums, then $\ef_l$ is perfect left filter and $_{\ef_l}R$ is perfect left ring of quotients by Theorem \ref{perfect_filter}. Thus, for every left $R$-ideal $I\in \ef_l,$ $_{\ef_l}Rq^l_R(I)=\,_{\ef_l}R.$ The isomorphism $j:\, _lR_r\cong\,_{\ef_l}R$ establishes that every left ideal $I$ is such that $_lR_r q_R(I)=\,_lR_r $ iff $_{\ef_l}Rq_R^l(I)=\,_{\ef_l}R.$ Thus, the claim in (3) about the left $R$-ideals holds. Similarly, the claim in (3) about the right $R$-ideals holds.

As $_{\ef_l}R$ is perfect, there is an $R$-bimodule isomorphism $\phi:\,_{\ef_l}R\otimes_R\, _{\ef_l}R\rightarrow \,_{\ef_l}R$ given by $\phi(s\otimes t)=st.$ Combining $\phi$ with the ring isomorphism $j:\,_lR_r\rightarrow\,_{\ef_l}R$ gives us the $R$-bimodule isomorphism $_lR_r\otimes_R\,_lR_r\cong\,_lR_r$ defined by $p\otimes q\mapsto j^{-1}(\phi(j(p)\otimes j(q)))=pq.$

Now, let $K$ be an $R$-ideal in $_l\ef_r.$ Consider a free resolution of $R/K,$ $R^{\alpha}\rightarrow R^{\beta}\rightarrow R/K\rightarrow 0$ for some ordinals $\alpha$ and $\beta.$ Since $q$ is exact and preserves direct sums, applying functors $i$ and $q$ to this exact sequence produces the following commutative diagram
{\scriptsize
\begin{diagram}
(_lR_r\otimes_R\,_lR_r)^{\alpha} & \rTo & (_lR_r\otimes_R\,_lR_r)^{\beta} & \rTo & {_lR_r\otimes_R R/K\otimes_R\,_lR_r} &\rTo & 0\\
\dTo_{\cong} & & \dTo_{\cong} & & \dTo & & \\
(_lR_r)^{\alpha} & \rTo & (_lR_r)^{\beta} & \rTo & {\,_l(R/K)_r} &\rTo & 0
\end{diagram}
}
From this diagram, it follows that the last vertical map is an $R$-bimodule isomorphism. As $K$ is in $_l\ef_r,$ $\,_l(R/K)_r=0.$ So, $_lR_r\otimes_R R/K\otimes_R\,_lR_r=0$ and thus $_lR_r q_R(K)\,_lR_r=\, _lR_r.$

$(1)+(4)\Rightarrow(5)$ Condition (1) implies the first sentence of (5) by Lemma \ref{lemma_perfect}. Functors $q^l$ and $q^r$ are exact and preserve direct sums by part (5) of Theorem \ref{perfect_filter}. As functor $i$ defined by $i_M: M\rightarrow\,_lR_r \otimes_R M\otimes_R\,_lR_r $ is exact and preserves direct sums if $_lR_r$ is perfect, part (4) implies that $q$ is exact and preserves direct sums.

$(6)\Rightarrow(3)$ If $q^l$ is exact and $\ef_l$ has a basis of finitely generated left ideals, then $\ef_l$ is perfect left filter and $_{\ef_l}R$ is perfect left ring of quotients. The isomorphism $j:\, _lR_r\cong\,_{\ef_l}R$ establishes that every left ideal $I$ is such that $_lR_r q_R(I)=\,_lR_r $ iff $_{\ef_l}Rq_R^l(I)=\,_{\ef_l}R.$ Thus, Theorem \ref{perfect_filter} establishes that $\ef_l=\{I|_lR_r q_R(I)=\,_lR_r\}.$ We obtain analogous claim for the right filter $\ef_r$ and $R_{\ef_r}.$ We also obtain the $R$-bimodule isomorphism $_lR_r\otimes_R\,_lR_r\cong\,_lR_r$ just as in proof of $(5)\Rightarrow(3).$

To finish the proof of (3), it is sufficient to prove that $_lR_r q_R(K)\,_lR_r=\,_lR_r$ for every $K\in\,_l\ef_r.$ Since $_l\ef_r$ has a basis of finitely generated ideals, it is sufficient to prove this claim for $K$ finitely generated. If $K$ is finitely generated, $R/K$ is finitely presented so there is an exact sequence $R^{m}\rightarrow R^{n}\rightarrow R/K\rightarrow 0$ for some nonnegative integers $m$ and $n.$ Since $q$ is exact, applying functors $i$ and $q$ to this exact sequence produces the following commutative diagram
{\scriptsize
\begin{diagram}
(_lR_r\otimes_R\,_lR_r)^{m} & \rTo & (_lR_r\otimes_R\,_lR_r)^{n} & \rTo & {_lR_r\otimes_R R/K\otimes_R\,_lR_r} &\rTo & 0\\
\dTo_{\cong} & & \dTo_{\cong} & & \dTo & & \\
(_lR_r)^{m} & \rTo & (_lR_r)^{n} & \rTo & {\,_l(R/K)_r} &\rTo & 0
\end{diagram}
}
From this diagram, it follows that the last vertical map is an isomorphism. As $K\in\,_l\ef_r,$ $_l(R/K)_r=0$. Then $_lR_r\otimes_R R/K\otimes_R\,_lR_r=0$ as well and so $_lR_r q_R(K)\,_lR_r=\,_lR_r.$

$(1)+(4)\Rightarrow(6)$ Condition (1) implies the first sentence of (6) by Lemma \ref{lemma_perfect}. Functors
$q^l$ and $q^r$ are exact and filters $\ef_l$ and $\ef_r$ have basis of finitely generated ideals by part (6) of Theorem \ref{perfect_filter}. In this case the symmetric filter generated by $\ef_l$ and $\ef_r$ have a finite basis also and part (4) implies that the functor $q$ is also exact.
\end{proof}

Theorem \ref{perfect_symmetric_filter} establishes a one-to-one correspondence between the
set of perfect induced filters $_l\ef_r$ on $R$ and the perfect symmetric rings of
quotients given by $_l\ef_r \mapsto\; _lR_r$ with the inverse $S\mapsto\; _l\ef_r$ where
$\ef_l=\{I| Sf(I)=S\}$ and $\ef_r=\{J| f(J)S=S\}$
for $f:R\rightarrow S$ epimorphism that makes $S$ into a
flat left and a flat right $R$-module. The theorem also establishes that a perfect filter $_l\ef_r$ is equal to the filter of the symmetric tensoring torsion theory $_l\tau_r=\,_{_lR_r}\tau_{_lR_r}.$

\section{Total symmetric ring of quotients}
\label{section_total_symmetric}

In this section we prove the existence of the maximal perfect symmetric ring of quotients (i.e symmetric version of Theorem XI 4.1, p. 233, \cite{Stenstrom}). To be consistent with the existing terminology for the right-sided version of this concept, we call it {\em total symmetric ring of quotients,} or maximal perfect symmetric localization of $R$ and denote it by $\Qsimtot(R).$
\begin{theorem}
For a right $R$, there exist a ring $\Qsimtot(R)$ and an injective ring epimorphism $f:R\rightarrow\Qsimtot(R)$ such that
\begin{enumerate}
\item $f$ makes $\Qsimtot(R)$ into left and right flat $R$-module;

\item If $g: R\rightarrow S$ is an injective ring epimorphism that makes $S$ into left and right flat $R$-module, then there is a unique monomorphism $\overline{g}: S\rightarrow\Qsimtot$ with $f=\overline{g}g.$
\end{enumerate}
\label{total_symmetric_quotient}
\end{theorem}
\begin{proof}
If $f_S:R\rightarrow S$ is a ring epimorphism that makes $S$ left and right $R$-flat, then $\ker f_S$ is the torsion submodule of $R$ with respect to  $_S\tau_S.$ If $f_S$ is injective, $_S\tau_S$ is faithful by part (4) of Proposition \ref{symmetric_tensor} and so $S$ can be embedded in $\Qsimmax(R).$ Thus, let us consider the family $\pe$ of all subrings $S$ of $\Qsimmax(R)$ such that an injective epimorphism $f_S: R\rightarrow S$ makes $S$ into a perfect symmetric ring of quotients of $R.$ This family is nonempty as $R$ is in it. The inclusion of subrings in $\pe$ corresponds to the inclusion in the perfect torsion theories they determine just as in the one-sided case (see the first paragraph in the proof of Theorem 4.1, p. 233 in \cite{Stenstrom}). Also,
for $S\in\pe,$ $f_S$ composed with the isomorphism $S\cong\,  _{_S\ef_S}R_{_S\ef_S}$ is the canonical homomorphism $R\rightarrow\,_{_S\ef_S}R_{_S\ef_S}$ given by $r\mapsto (R_r , L_r )$ where $R_r$ is the multiplication by $r$ from the right and $L_r$ is the multiplication by $r$ from the left. Thus, if $S$ and $T$ are in $\pe,$ $i: S\rightarrow T$ the inclusion of $S$ in $T,$ and $f_T: R\rightarrow T$ injective ring epimorphism that makes $T$ into a left and right flat $R$-module, it is easy to see that both $if_S$ and $f_T$ map any $r\in R$ to element of $T$ that corresponds to $(R_r, L_r)$ under the isomorphism $T\cong \,_{_T\ef_T}R_{_T\ef_T}.$ Thus, $f_T=i f_S.$

Also, if $S$ and $T$ are in $\pe,$ then the smallest subring $U$ of $\Qsimmax(R)$ that contains $S$ and $T$ satisfies condition (3) of Theorem \ref{PerfectQuotient} by Lemma 4.2, p. 234 in \cite{Stenstrom}. By symmetry, the condition (3) of Theorem \ref{PerfectSymmetricQuotient} is satisfied so $U$ is perfect symmetric. As elements of $U$ are of the form $\sum_i\prod_j s_{ij} t_{ij},$ the injectivity of $f_S$ and $f_T$ guarantees the injectivity of $f_U.$ So, $U$ is in $\pe.$ Thus, $\pe$ is directed by inclusion.

Consider the direct limit of all elements of $\pe$. Denote it by $\Qsimtot(R).$ We define the injective map $f: R\rightarrow \Qsimtot(R)$ as $f=i_Sf_S$ for any perfect symmetric ring of quotients $S$ with ring epimorphism $f_S: R\rightarrow S$ and the embedding $i_S: S\rightarrow\Qsimtot(R).$ $f$ is well defined as if $T\in\pe$ is an overring of $S$ with $i$ inclusion of $S$ in $T,$ then $i_Tf_T=i_T i f_S=i_S f_S=f.$ Clearly, $f$ is injective. To show that $f$ is a ring epimorphism, consider the ring homomorphisms $\alpha, \beta: \Qsimtot(R)\rightarrow T$ for any ring $T$ such that $\alpha f=\beta f.$ As $f_S$ is epimorphism for every $S\in \pe,$ $\alpha i_S f_S=\beta i_S f_S$ implies that $\alpha i_S=\beta i_S.$ But as this holds for every $S\in\pe$, it implies that $\alpha=\beta$ on entire $\Qsimtot(R).$ $\Qsimtot(R)$ is left and right flat as the direct limit of flat modules is flat. Thus, $\Qsimtot(R)$ is the maximal element of $\pe.$

If $g: R\rightarrow S$ is an injective ring epimorphism that makes $S$ into left and right flat $R$-module, then $S$ is in $\pe$ and so it embeds in $\Qsimtot(R).$ Denote this embedding by $\overline{g}.$ By construction of $\Qsimtot(R),$ $f=\overline{g}g.$ $\overline{g}$ is unique such map because $g$ is an epimorphism and so $f=\overline{g}g=hg$ implies $\overline{g}=h$ for any $h$ with $f=hg.$
\end{proof}

The next two corollaries prove some further properties of $\Qsimtot.$
\begin{corollary} $\Qsimtot(R)$ embeds in $\Qtot(R)$ and $\Qlrtot^l(R).$ If $\Qtot(R)=\Qlrtot^l(R),$ then $\Qsimtot(R)=\Qtot(R)=\Qlrtot^l(R).$
\label{corollary_of_Qsimtot_constuction}
\end{corollary}
\begin{proof}
Since the inclusion $R\subseteq\Qsimtot(R)$ is an epimorphism that makes $\Qsimtot(R)$ left $R$-flat, $\Qsimtot(R)\subseteq\Qtot(R).$ Similarly, $\Qsimtot(R)\subseteq\Qlrtot^l(R).$

As $\Qtot(R)$ is left $R$-flat,  $\Qlrtot^l(R)$ is right $R$-flat, and the inclusion $R\subseteq\Qtot(R)$ is an epimorphism, $\Qtot(R)=\Qlrtot^l(R)$ implies that $\Qtot(R)$ is a perfect symmetric ring of quotients. Thus, $\Qtot(R)=\Qlrtot^l(R)\subseteq\Qsimtot(R).$ The converse always holds.
\end{proof}

\begin{corollary}
(1) If $R$ is a left and right Ore ring, then $\Qlrcl(R)$ is the perfect symmetric ring of quotients and is contained in $\Qsimtot(R).$ The torsion theory $_{\Qlrcl(R)}\tau_{\Qlrcl(R)}$ coincides with the classical
torsion theory.

(2) If $R$ is von Neumann regular, then $R=\Qsimtot(R)=\Qtot(R)=\Qlrtot^l(R).$

(3) If $R$ is a right semihereditary ring with $\Qmax(R)=\Qlrtot^l(R),$ then $R$ is also left semihereditary and $\Qsimtot(R)=\Qsimmax(R)=\Qmax(R).$

(4) If $R$ is left and right nonsingular with $\Qmax(R)=\Qlmax(R)$ semisimple, then $\Qsimmax(R)=\Qmax(R)=\Qlmax(R)=\Qtot(R)=\Qlrtot^l(R)=\Qsimtot(R).$
\label{Qsimtot_of_special_rings}
\end{corollary}
\begin{proof}
(1) If $R$ is right Ore, the embedding $R\rightarrow\Qcl$ is a ring epimorphism that makes $\Qcl$ into a flat left $R$-module (Example on p. 230 and Proposition II.3.5 in \cite{Stenstrom} or Exercises 17 and 18, p. 318 in \cite{Lam}). So, $\Qcl$ is right perfect. Similarly, $\Qlrcl^l$ is left perfect if $R$ is left Ore.

If $R$ is left and right Ore, $\Qlrcl=\Qcl=\Qlrcl^l$ so $\Qlrcl$ is both left and right perfect. Thus, it is perfect symmetric by Theorem \ref{PerfectSymmetricQuotient}. As it contains $R,$ $\Qlrcl\subseteq\Qsimtot.$ Also, $\Qlrcl$ is the perfect symmetric ring of quotients of theory $_{\Qlrcl}\tau_{\Qlrcl}$ by Theorem \ref{perfect_symmetric_filter}.

An element $m$ of an $R$-bimodule $M$ is in the torsion submodule for the classical torsion theory iff $ms=0$ and $tm=0$ for some regular elements $s$ and $t.$ However, since the classical right (left) torsion theory is equal to the right (left) tensoring torsion theory $\tau_{\Qlrcl}$ ($\,_{\Qlrcl}\tau$), this condition is equivalent with $0=m\otimes 1\in M\otimes_R \Qlrcl$ and $0=1\otimes m\in \Qlrcl\otimes_R M.$ Part (2) of Proposition \ref{symmetric_tensor} shows that this is equivalent to $m\in\;_{\Qlrcl}\te_{\Qlrcl}(M).$ Thus, the classical torsion theory coincides with $_{\Qlrcl}\tau_{\Qlrcl}.$

(2) If $R$ is regular, then $R=\Qtot(R)=\Qlrtot^l(R)$ by Example 1, p. 235 in \cite{Stenstrom}. Thus $R=\Qsimtot(R)$ by part (2) of Corollary \ref{corollary_of_Qsimtot_constuction}.

(3) By Corollary 7.4, p. 259 in \cite{Stenstrom}, the assumptions imply that $R$ is left semihereditary and that $\Qmax(R)=\Qlrtot^l(R)=\Qtot(R).$ By part (2) of Corollary \ref{corollary_of_Qsimtot_constuction}, $\Qmax(R)$ is equal to $\Qsimtot(R)$ also. As $\Qsimtot(R)\subseteq\Qsimmax(R)\subseteq\Qmax(R),$ we have that $\Qsimtot(R)=\Qsimmax(R)=\Qmax(R)$ in this case.

(4) $\Qmax(R)$ is semisimple, then $\Qtot(R)=\Qmax(R)$ by Proposition 5.3, p. 236 and Example 1, p. 237 in \cite{Stenstrom}. Similarly, $\Qlrtot^l(R)=\Qlmax(R).$ $\Qmax(R)=\Qlmax(R)$ then implies
$\Qsimmax(R)=\Qmax(R)=\Qlmax(R)=\Qtot(R)=\Qlrtot^l(R)=\Qsimtot(R).$
\end{proof}

\begin{example}
(1) Let $F$ be a field and $R=\left(\begin{array}{lc}F\;\; & F\oplus F\\ 0\;\; & F \end{array}\right).$ Consider the maps
$f_l: R\rightarrow M_3(F)$ and $f_r: R\rightarrow M_3(F)$ defined by
\[f_l:\left(\begin{array}{lc}a\;\; & (x,y)\\ 0\;\; & b \end{array}\right)\mapsto
\left(\begin{array}{ccc}a & x & y\\ 0 & b & 0\\ 0 & 0 & b\end{array}\right)\;\;\;\;\;\;\; f_r:\left(\begin{array}{lc}a\;\; & (x,y)\\ 0\;\; & b \end{array}\right)\mapsto
\left(\begin{array}{ccc}a & 0 & x\\ 0 & a & y\\ 0 & 0 & b\end{array}\right)
\]
These are inclusions of $R$ in $\Qlmax(R)$ and $\Qmax(R)$ respectively (see Exercise 8, p. 260 in \cite{Stenstrom}) such that $R$ is not dense as a left $R$-submodule of $\Qmax(R)$ and as a right submodule of $\Qlmax(R)$. Thus $\Qlmax(R)\cong\Qmax(R)$ but they do not coincide as overrings of $R.$ Since $M_3(F)$ is a semisimple ring,
$\Qtot(R)=\Qmax(R)$ and $\Qlrtot^l(R)=\Qlmax(R)$ (see Example 1, p. 237 and Proposition 5.3, p. 236 in \cite{Stenstrom}). The maximal symmetric ring of quotients is equal to $R$ (Example 2.26, p. 31 in \cite{Ortega_thesis}). Also, as $R$ can be represented as a path algebra of the quiver with two vertices 1 and 2 and two arrows both going from 1 to 2, the maximal symmetric ring of quotients of $R$ can be computed using the techniques from \cite{Ortega_paper_Qsimmax} or chapter 2 of \cite{Ortega_thesis}. $\Qsimmax(R)=R$ implies that $\Qsimtot(R)=R.$ Thus, we have that
\[\Qsimtot\subsetneq\Qlrtot^l,\;\;\; \Qsimtot\subsetneq\Qtot,\;\;\; \Qtot\neq\Qlrtot^l,\;\;\; \Qtot\cong\Qlrtot^l.\]

(2) Let $R=\left(\begin{array}{cccc} F & F & F & F\\ 0 & F & 0 & F\\ 0 & 0 & F & F\\ 0 & 0 & 0 & F\\ \end{array}\right).$ This ring provides another example of the situation described in previous example. $R$ has $\Qlmax(R)$ and $\Qmax(R)$ both isomorphic to $M_4(F)$ but not equal as overrings of $R.$ $\Qsimmax(R)$ is equal to $R$ (see Example 2.1 d), p. 18 in \cite{Ortega_thesis}). Thus $R=\Qsimtot(R).$ Since $M_4(F)$ is a semisimple ring, $\Qtot(R)=\Qmax(R)$ and $\Qlrtot^l(R)=\Qlmax(R).$

(3) The examples of rings with $R\subsetneq \Qsimtot(R)=\Qsimmax(R)$ can be found among the rings in the
class $\ce$ of Baer *-rings considered in \cite{Be2} and \cite{Lia2}.
A ring $R$ from $\ce$ has left and right maximal and classical
rings of quotients equal by Proposition 3 in \cite{Lia2}. Thus, $\Qsimtot(R)$ is equal to those quotients as well.
In this case, $\Qmax=\Qlmax=\Qsimmax=\Qtot=\Qlrtot^l=\Qsimtot=\Qlrcl.$  Group von Neumann algebras
are in class $\ce.$ A group von Neumann algebra $N(G)$ for $G$ infinite has maximal rings of quotients that is not semisimple and strictly contains the algebra $N(G)$ (see Exercise 9.10 in \cite{Luck}). So $N(G)\subsetneq \Qsimtot(N(G))$ in this case.
\label{Examples_of_total_symmetric}
\end{example}

\section{Adapting Morita's construction of $\Qtot$ to a construction of $\Qsimtot$}
\label{section_Morita_symmetric}

The total right ring of quotients $\Qtot(R)$ is usually
obtained as the directed union of the family of all subrings of
$\Qmax(R)$ that are perfect right rings of quotients of $R$ (see Theorem XI 4.1, p. 233, \cite{Stenstrom}).
In \cite{Morita3}, Morita constructs
$\Qtot(R)$ differently. The simplification of Morita's construction for a class of rings is
given in \cite{Lia4}.

In Theorem 3.1 of \cite{Morita3}, Morita showed that a ring
homomorphism $f:R\rightarrow S$ is a ring epimorphism with $S$
flat as a left $R$-module iff $S$ is the right ring of
quotients of $R$ with respect to the Gabriel filter $\ef_t(S)=\{I$ right ideal $|(f(r):f(I))S=S$ for every $r\in R\}.$
In this case $S=\{s\in S | (sf(r): f(R))S=S$ for every $r\in R\}.$
Motivated by this result, Morita considered the set
\[S'=\{s\in S | (sr: R)S=S\mbox{ for every }r\in R\}\]
for a ring extension $S$ of $R.$ By Theorem 3.1 of \cite{Morita3}, $S$ is a flat
epimorphic extension iff $S=S'.$ In Lemma 3.2 of
\cite{Morita3}, Morita proved that $S'$ is a subring of $S$ that
contains $R.$ Then he defined a family of subrings of $S$ indexed by ordinals on the following way: $S^{(0)}=S,$  $S^{(\alpha+1)}=(S^{(\alpha)})'$ and $S^{(\alpha)}=\bigcap_{\beta<\alpha}
S^{(\beta)}$ if $\alpha$
is a limit ordinal. Finally, Morita proved that there is an ordinal $\gamma$ such
that $S^{(\gamma)}=(S^{(\gamma)})'=S^{(\gamma+1)}.$ Thus, this construction produces the largest flat epimorphic extension of $R$
contained in a given extension $S.$ By taking $S=\Qmax(R),$ Morita arrived to $\Qtot(R).$

In this section, we modify this construction to the construction of the total symmetric ring of quotients. For a ring $R$ and its extension $S$, define \[S'=\{s\in S\; | \; S(R:rs)=S\mbox{ and }(sr:R)S=S\mbox{ for every }r\in R\}.\]
\begin{proposition}
\begin{enumerate}
\item $S'$ is a subring of $S$ that contains $R.$

\item $S'$ is a symmetric ring of quotients with respect to the torsion theory induced by $\ef_l=\{I| S(I:r)=S$ for all $r\in R \}$ and $\ef_r=\{J| (r:J)S=S$ for all $r\in R\}.$

\item $S$ is perfect symmetric ring of quotients if and only if $S'=S.$

\item Define
\[S^{(0)}=S,\;\;S^{(\alpha+1)}=(S^{(\alpha)})',\mbox{ and } S^{(\alpha)}=\bigcap_{\beta<\alpha}S^{(\beta)}\mbox{ for }\alpha\mbox{ a limit ordinal}.\]
Then there is an ordinal $\gamma$ such
that $S^{(\gamma)}=S^{(\gamma+1)}.$

\item If $R\subseteq T\subseteq S$ and the inclusion $R\subseteq T$ is a left and right flat epimorphism, then $T\subseteq S'.$

\item If $S$ is left and right $R$-flat, then $S'$ is the symmetric ring of quotients with respect to the tensoring torsion theory $_S\tau_S.$
\end{enumerate}
\label{Morita_for_symmetric}
\end{proposition}
\begin{proof}
(1) Denote $S'_l=\{s\in S| S(R:rs)=S\mbox{ for every }r\in R\}$ and $S'_r=\{s\in S| (sr:R)S=S\mbox{ for every }r\in R\}.$ From Morita's construction (Lemma 3.2 in \cite{Morita3}) the sets $S'_l$ and $S'_r$ are subrings of $S$ that contains $R$. Then $S'=S'_l\cap S'_r$ is a subring of $S$ that contains $R.$

(2) Note that $\ef_l$ and $\ef_r$ are Gabriel filters by Lemma 1.1 in \cite{Morita3}.
Let $_lR_r$ be the symmetric ring of quotients with respect to the torsion theory induced by filters $\ef_l$ and $\ef_r.$ Note that $R$ is torsion-free for this torsion theory. Indeed, if $r\in\;_l\te_r(R),$ then the left annihilator $\ann_l(r)$ is in $\ef_l$ and the right annihilator $\ann_r(r)$ is in $\ef_r.$ Thus $S(\ann_l(r):1)=S$ and so $1=\sum_{i=1}^n s_ir_i$  for some $s_i\in S$ and $r_i\in\ann_l(r).$ But then $r=1r=\sum s_i r_i r=\sum s_i 0=0.$

We show now that $_lR_r=S'.$ Let $s\in\,_lR_r.$ As $R$ is torsion-free, there are $I\in \ef_l$ with $Is\subseteq R$ and $J\in \ef_r$ with $sJ\subseteq R.$ As $Is\subseteq R,$ $I$ is contained in $(R:s).$ As $I\in \ef_l,$ $(R:s)$ is in $\ef_l.$ But that means that $S((R:s):r)=S$ for every $r\in R.$ As $(R: rs)=((R:s):r),$ $S(R: rs)=S.$
Similarly, we prove that $(sr:R)S=S$ using that $(s:R)$ is in $\ef_r.$ Then $s$ is in $S'.$

Conversely, if $s\in S',$ then $(R:s)\in \ef_l$ as $S=S(R: rs)=S((R:s):r)$ for every $r.$ Similarly, $(s:R)\in \ef_r$ as $(sr:R)=(r:(s:R)).$ Then $s$ defines a pair of compatible homomorphisms $R_s: (R:s)\rightarrow R$ and $L_s: (s:R)\rightarrow R.$ Thus, $s\in\,_lR_r.$

(3) Theorem 3.1 in \cite{Morita3} states that the inclusion $R\subseteq S$ is a left flat epimorphism iff  $S'_l=S.$ Thus, if the inclusion $R\subseteq S$ is a left and right flat epimorphism, then $S=S'_r$ and $S=S'_l.$ But then $S=S'_l\cap S'_r=S'.$ Conversely, if $S=S',$ then $S=S'_l=S'_r$ and so the inclusion $R\subseteq S$ is a left and right flat epimorphism.

(4) If $S^{(\gamma+1)}$ is strictly contained in $S^{(\gamma)}$ for every $\gamma,$ then $|S|\geq |S-S^{(\gamma)}|\geq |\gamma|$ for every ordinal $\gamma.$ This is a contradiction, so there is $\gamma$ with $S^{(\gamma)}=S^{(\gamma+1)}.$

(5) $T\subseteq S$ implies $T'\subseteq S'.$ If $T$ is perfect, $T'=T$ by part (3). Thus, $T\subseteq S'.$

(6) If $S$ is left and right flat, the sets of left and right ideals $_S\ef=\{I| SI=S\}$ and $\ef_S=\{J| JS=S\}$ are Gabriel filters (note that we need the condition that $S$ is flat to have that corresponding torsion theories are hereditary) and they induce $_S\tau_S$ by part (3) of Proposition \ref{symmetric_tensor}.
As $S'$ is the symmetric ring of quotients with respect to $\ef_l$ and $\ef_r$ by part (2), it is sufficient to prove that $_S\ef=\ef_l$ and $\ef_S=\ef_r.$ Let $I\in\ef_l.$ Then $S(I:1)=S$ so $SI=S.$ Conversely, if $SI=S,$ then $S(I:r)=S$ for all $r\in R$ by property (1) of Gabriel filters (see section \ref{section_right_quotients}).
\end{proof}

\begin{theorem} If $\gamma$ is the ordinal for which $\Qsimmax(R)^{(\gamma)}=\Qsimmax(R)^{(\gamma+1)}$ (and such ordinal exists by part (4) of Proposition \ref{Morita_for_symmetric}) then $\Qsimtot(R)=\Qsimmax(R)^{(\gamma)}.$
\label{Moritas_construction}
\end{theorem}
\begin{proof} By part (5) of Proposition \ref{Morita_for_symmetric}, $\Qsimtot(R)\subseteq\Qsimmax(R)^{(\alpha)}$ for all ordinals $\alpha.$ By part (4) of Proposition \ref{Morita_for_symmetric}, there is an ordinal $\gamma$ such that $\Qsimmax(R)^{(\gamma)}=\Qsimmax(R)^{(\gamma+1)}.$ Then, $\Qsimmax(R)^{(\gamma)}$ is a perfect symmetric ring of quotients by part (3) of Proposition \ref{Morita_for_symmetric}. As $\Qsimtot(R)$ is the largest of all perfect symmetric rings of quotients that contain $R,$ $\Qsimmax(R)^{(\gamma)}\subseteq\Qsimtot(R).$ Thus, $\Qsimmax(R)^{(\gamma)}=\Qsimtot(R).$

\end{proof}

\begin{example}
(1) Rings from class $\ce$ in part (3) of Examples \ref{Examples_of_total_symmetric} are examples of semihereditary rings for which $\Qmax=\Qlmax=\Qsimmax=\Qtot=\Qlrtot^l=\Qsimtot$ but the two-sided tensoring with $\Qsimmax=\Qmax$ does not define a perfect torsion theory (see Example 4.1 in \cite{Lia4}).

(2) The ring in Example 4.2 in \cite{Lia4} provides an example of a semihereditary ring $R$ with $\Qtot(R)=\Qmax(R)'_r\subsetneq \Qmax(R).$ As this ring is commutative, this is also an example of a ring with $\Qsimtot(R)=\Qsimmax(R)'\subsetneq\Qsimmax(R).$
\end{example}

\section{Questions and open problems}
\label{section_questions}

(1) Can the first sentence in parts (4), (5) and (6) of Theorem \ref{perfect_symmetric_filter} be deleted? If so, Theorem \ref{perfect_symmetric_filter} would completely parallel Theorem \ref{perfect_filter}.

(2) Both condition $\Qmax=\Qtot$ and condition $\Qmax=\Qlrtot^l$ have been studied in the past. In the first case, the study led to consideration of rings $R$ with $\Qmax(R)$ being right Kasch. In the second case, the study produced results of Goodearl (Thm. 7.1, p. 257 in \cite{Stenstrom}), Cateforis (Cor. 7.4, p. 259 in \cite{Stenstrom}), Evans (Thm. 2.4 in \cite{Evans}) and others. Corollary \ref{Qsimtot_of_special_rings} scratches the surface of the study of rings with $\Qsimmax=\Qsimtot.$

(3) In \cite{Ortega_paper_Qsimmax} and \cite{Ortega_thesis}, algorithms for calculating $\Qsimmax,$ $\Qmax$ and $\Qlmax$ of finite-dimensional incidence and path algebras are given in terms of the underlying partially ordered sets and finite quivers. It would be interesting to see if similar algorithms for calculating $\Qsimtot,$ $\Qtot$ and $\Qlrtot^l$ of incidence and path algebras exist.


\begin{thebibliography}{10}
\bibitem{Be2} S. K. Berberian, Baer $*$-rings,
Die Grundlehren der mathematischen Wissenschaften 195,
Springer-Verlag, Berlin-Heidelberg-New York, 1972.

\bibitem{Bland_book} P. E. Bland, Topics in torsion theory, Math.
Research 103, Wiley-VCH Verlag Berlin GmbH, Berlin, 1998.

\bibitem{Evans} M. W. Evans, A class of semihereditary rings,
Rings, modules and radicals (Hobart, 1987), 51--60, Pitman Res.
Notes Math. Ser., 204, Longman Sci. Tech., Harlow, 1989.

\bibitem{Gabriel} P. Gabriel, Des cat\'egories ab\'eliennes, Bull.
Soc. Math. France 90 (1962) 323--448.

\bibitem{Lam} T. Y. Lam,  Lectures on modules and rings,
Graduate Texts in Mathematics, 189, Springer-Verlag, New York,
1999.

\bibitem{Lanning} S. Lanning, The maximal symmetric ring of quotients, J. of Algebra 179 (1996) 47--91.

\bibitem{Luck} W. L\"{u}ck: $L^2$-invariants: Theory and Applications
to Geometry and K-theory. Ergebnisse der Mathematik und ihrer
Grebzgebiete, 44 (3), Sprin\-ger-Verlag, Berlin, 2002.

\bibitem{Morita3} K. Morita, Flat modules, injective modules and quotient
rings, Math. Z.  120 (1971) 25--40.

\bibitem{Ortega_paper} E. Ortega, Two-sided localization of bimodules, Communications in Algebra 36 (2008) no. 5, 1911--1926.

\bibitem{Ortega_paper_Qsimmax} E. Ortega, Rings of quotients of incidence algebras and path algebras, J. of Algebra  303 (2006) no. 1, 225--243.

\bibitem{Ortega_thesis} E. Ortega, The maximal symmetric ring of quotients: path algebras, incidence algebras and bicategories, Dissertation, Universitat Aut\`{o}noma de Barcelona (2006).


\bibitem{Schelter} W. Schelter, Two-sided rings of quotients, Arch. Math., 24 (1973) 274--277.

\bibitem{Stenstrom} B. Stenstr\"{o}m, Rings of quotients, Die Grundlehren der Mathematischen
Wissenschaften 217, Springer-Verlag, New York-Heidelberg (1975).

\bibitem{Utumi} Y. Utumi, On rings of which one-sided quotient ring are two-sided, Proc. Amer. Mah. Soc. 14 (1963) 141--147.

\bibitem{Lia2} L. Va\v s, Dimension and Torsion Theories for a Class of Baer
*-Rings, J. of Algebra  289 (2005) no. 2, 614--639.

\bibitem{Lia4} L. Va\v s, A Simplification of Morita's Construction of Total Right
Rings of Quotients for a Class of Rings, J. of Algebra 304 (2006) no. 2, 989--1003.
\end{thebibliography}
\end{document}